\documentclass[a4paper,11pt]{article}
\author{Aijin Lin}
\title{$\textmd{The heat flow for K}$$\ddot{\textmd{a}}$$\textmd{hler fibrations}$}
\date{}
\usepackage{amsfonts}
\usepackage{amsthm}
\begin{document}
\maketitle
\begin{abstract}
We establish global existence of smooth solutions to heat flow for Yang-Mills-Higgs functional on K$\ddot{\textmd{a}}$hler fibrations. As an application, we give a new proof of the key inequality for Mundet's Hitchin-Kobayashi correspondence theorem using the heat flow technique.
\end{abstract}
\vspace{12pt}
\let\thefootnote\relax \footnotetext {The author's name Aijin Lin was earlier known as Ling Lin. I am deeply sorry for any inconvenience that I have brought to you.}
\newtheorem{theorem}{Theorem}[section]
\newtheorem{lemma}[theorem]{Lemma}
\newtheorem{proposition}[theorem]{Proposition}
\newtheorem{corollary}[theorem]{Corollary}
\section{Introduction}
$\hspace{15pt}$ Assume that $X$ is a compact K$\ddot{\textmd{a}}$hler manifold, $G$ is a compact connected Lie group, $\mathfrak{g}=Lie(G)$, and $P\rightarrow X$ is a principle $G$-bundle on $X$. Now suppose that $(M, \omega, \mu)$ is a compact symplectic manifold which supports a Hamiltonian $G$-group action with a moment map $\mu: M\rightarrow \mathfrak{g}^{\ast}$, and fix on $\mathfrak{g}$ an invariant inner product $\langle\cdot,\cdot\rangle$. This allows to identify $\mathfrak{g}\simeq \mathfrak{g}^{\ast}$. Let $\mathcal{A}$ be the space of connections on $P$, $\mathcal{A}^{1,1}\subset \mathcal{A}$ be the space of connection whose curvature belongs to $\Omega^{1,1}(P\times_{Ad}\mathfrak{g})$. Let $\mathcal{F}=P\times_{G}M\rightarrow X$ be the associated fiber bundle (fibration) on $X$, $\mathcal{L}$ the space $\Gamma(\mathcal{F})$ of smooth sections of $\mathcal{F}$. \par

In 1999, Cieliebak-Gaio-Salamon [5] and Mundet [13] independently defined analogues of the Gromov-Witten invariants of symplectic manifolds with Hamiltonian group actions using the vortex equation
\begin{eqnarray}
\wedge F_{A}+\mu(u)=c,
\end{eqnarray}
where $F_{A}$ denotes the curvature of the connection $A$ on $P$, $\wedge$ the contraction with the K$\ddot{\textmd{a}}$hler form of $X$, $u\in\mathcal{L}$, and  $c\in \mathfrak{g}$ is a fixed central element.
They also introduced the Yang-Mills-Higgs functional (see section 2)
\begin{equation}
\mathcal{YMH}(A,u)=||F_{A}||_{L^{2}}^{2}+||d_{A}u||_{L^{2}}^{2}+||\mu(u)-c||_{L^{2}}^{2},
\end{equation}
and prove its some fundamental properties. When $X=\mathbb{R}^{2}$, the functional appeared in the Ginzburg-Landau theory in particle physics, which is used to model superconductivity. When $X=\mathbb{R}^{3}, G=SU(2)$ and $P=\mathbb{R}^{3}\times SU(2)$ is a trivial bundle, Taubes [18] studied the existence of critical points of Yang-Mills-Higgs functional. When $X=\Sigma$ is a compact Riemann surface and $M$ is a compact symplectic manifold, Song [17] proved the existence of critical points of Yang-Mills-Higgs functional above using Sacks-Uhlenbeck's perturbation method. \par
On the other hand, the most interesting property of the Yang-Mills-Higgs functional is that the vortex equation (1) is just the minimum equation of Yang-Mills-Higgs functional (2). And (1) has been studied in a large number of work, for example [3-5, 9, 13-15]. Especially, Bradlow [4] established the relation between the $\tau$-stability of holomorphic bundles with a global section and the existence of a hermitian metric satisfying the $\tau$-vortex equations by minimizing the so called Donaldson's functional. In [13, 14], Mundet proved the Hitchin-Kobayashi correspondence between the $c$-stability and existence of solution to the vortex equation (1). Hong [8] gave another proof of Bradlow's Hitchin-Kobayashi correspondence theorem by using heat flow approach, and his proof was considerably simpler than one in [4].\par
Motivated by Hong's work, we investigate the Yang-Mills-Higgs functional (2) using heat flow approach. Recently, Venugopalan [21] studied the heat flow of a Yang-Mills-type functional similar to (2) on the space of gauged holomorphic maps using another approach. \par Now let's describe the outline of this paper. First we will give the generalized covariant derivative $d_{A}$ (see section 2) introduced in [5]
\begin{eqnarray*} d_{A}u=du+L_{u}A, \forall u\in \mathcal{L}=\Gamma(\mathcal{F}), \end{eqnarray*}
and strictly define the Yang-Mills-Higgs functional
\begin{eqnarray*}
\mathcal{YMH}(A,u)=||F_{A}||_{L^{2}}^{2}+||d_{A}u||_{L^{2}}^{2}+||\mu(u)-c||_{L^{2}}^{2}.
\end{eqnarray*}
Then we adopt Urakawa's approach [20] and strictly deduce the Euler-Lagrange equations of Yang-Mills-Higgs functional as follows:
\begin{eqnarray}
L^{*}_{u}d_{A}u+D^{*}_{A}F_{A}&=0\\
\nabla^{*}_{A}d_{A}u+(d\mu)^{*}(\mu(u)-c)&=0
\end{eqnarray}

We can see that the Euler-Lagrange equations above are common generalization of harmonic maps and Yang-Mills connections.
Therefore, we are interested in studying the Euler-Lagrange equations and consider the following heat flow equations
\begin{eqnarray}
\frac{\partial A}{\partial t}&=&-L^{*}_{u}d_{A}u-D^{*}_{A}F_{A};\\
\frac{\partial u}{\partial t}&=&-\nabla^{*}_{A}d_{A}u-(d\mu)^{*}(\mu(u)-c).
\end{eqnarray}
And we can prove the global existence of (5), (6). Now we can state our main theorem as follows:
\begin{theorem} Assume that both the base manifold $X$ and the fiber $M$ of the fiber bundle $\mathcal{F}=P\times_{G}M\rightarrow X$ are compact K$\ddot{\textmd{a}}$hler manifolds (K$\ddot{\textmd{a}}$hler fibration), where $M$ supports a Hamiltonian $G$-group action, and $P$ is a holomorphic G-principal bundle over $X$. Let $H_{0}$ be a Hermitian metric on $P$, $A_{0}$ the canonical connection with respect to $H_{0}$, $u_{0}$ a holomorphic section of the fiber bundle $P\times_{G}M$, i.e. ${\bar \partial}_{A_{0}}u_{0}=0$.
Then Yang-Mills-Higgs heat flow
\begin{eqnarray*}
\frac{\partial A}{\partial t}&=&-L^{*}_{u}d_{A}u-D^{*}_{A}F_{A};\\
\frac{\partial u}{\partial t}&=&-\nabla^{*}_{A}d_{A}u-(d\mu)^{*}(\mu(u)-c),
\end{eqnarray*}
admits a unique global smooth solution $(A, u)$ in $X\times [0, +\infty)$ with initial value \begin{eqnarray} (A(\cdot, 0)=A_{0}, u(\cdot,0)=u_{0}). \end{eqnarray}
\end{theorem}
For the proof of theorem, we follow Donaldson's approach [6, 8] to consider a flow of gauge transformation which is equivalent to (5),(6):
 \begin{eqnarray*}
\frac{\partial h}{\partial t}=-\triangle_{A_{0}}h-i[h\hat{F}_{A_{0}}+\hat{F}_{A_{0}}h+2(\mu(u_{0})-c)h]+2i\wedge({\bar \partial}_{A_{0}}h h^{-1}{\partial}_{A_{0}}h)\\ \end{eqnarray*}

with $h(0)=I$. After that, we introduce the generalized K$\ddot{\textmd{a}}$hler identities and then prove some fundamental lemmas. Then modifying the idea of Donaldson, we obtain a local and global existence of the unique solution to (7) with initial value $h(0)=I$.\par
Next we recall the definition of $c$-stability of [14]. Then by studying the limiting behavior of the heat flow as $t\rightarrow\infty$, we can prove the key inequality [14, lemma 6.1] which is critical to the proof of the sufficiency part of Hitchin-Kobayashi correspondence theorem as follows:
\begin{theorem}([14]) Let $(A, u)\in \mathcal{A}^{1,1}\times \mathcal{L}$ be a simple pair, and assume that $(A, u)$ is c-stable, then there exists positive constants $C_{1}, C_{2}$ such that for any $s\in {Met}^{p}_{2,B}$, one has
\begin{equation}
sup|s|\leq C_{1}\Psi^{c}(e^{s})+C_{2},
\end{equation}
where the definitions of ${Met}^{p}_{2,B}$ and $\Psi^{c}$ can be seen in section 4.
\end{theorem}
This paper is organized as follows. In section 2, we define Yang-Mills-Higgs functional and strictly deduce its Euler-Lagrange equations. In section 3, we study the heat flow of Yang-Mills-Higgs functional and prove Theorem 1.1. As an application, we give another proof of Theorem 1.2 using the heat flow technique.
\section{Preliminaries}
\subsection{ Yang-Mills-Higgs functional}
$\hspace{15pt}$Suppose that $X$ is a compact K$\ddot{\textmd{a}}$hler manifold, G is a compact connected Lie group, and $P\rightarrow X$ is a principle $G$-bundle on $X$. Let $(M, \omega)$ be a compact symplectic manifold, and let $G$ acts on $M$ by symplectomorphisms. Let $\mathfrak{g}$=Lie(G) denote the Lie algebra of $G$. For every $\xi\in \mathfrak{g}$, denote by $X_{\xi}: M\rightarrow TM$ the vector field whose flow is given by the action of the 1-parameter subgroup generated by $\xi$. Suppose that the Lie algebra $\mathfrak{g}$ carries an invariant inner product $\langle\cdot,\cdot\rangle$. The action of $G$ is called $Hamiltonian$ if there exists an equivariant map $\mu: M\rightarrow \mathfrak{g}$ such that, for every $\xi\in\mathfrak{g}$,
\begin{eqnarray*} d\langle\mu, \xi\rangle=l(X_{\xi})\omega. \end{eqnarray*}
This means that $X_{\xi}$ is the Hamiltonian vector field of the function $\langle\mu,\xi\rangle$. The map $\mu$ is called a $moment$ $map$.\par
Now suppose that $(M, \omega, \mu)$ is a symplectic manifold with a Hamiltonian group $G$-action. Denote by $\mathcal{J}(M, \omega, \mu)$ the space of all almost complex structures $J$ on $TM$ which are invariant under the $G$-action and compatible with the symplectic structure $\omega$; that is , $\omega$ is $J$-invariant, and $g(\cdot,\cdot)=\omega(\cdot,J\cdot)$ is a Riemannian metric on $M$. It follows from [12, Prop. 2.50] that the space $\mathcal{J}(M, \omega, \mu)$ is nonempty and contractible.
Define an operator $L: M\rightarrow C^{\infty}(\mathfrak{g}, TM)$ as $L_{x}=L(x): \mathfrak{g}\rightarrow T_{x}M, \forall x\in M$, where $L_{x}\xi:=X_{\xi}(x), \forall \xi\in \mathfrak{g}$.
\begin{lemma}Suppose that $(M, \omega, \mu)$ is a symplectic manifold with a Hamiltonian group $G$-action.
Take a $\omega$-compatible almost complex structure $J$ on $M$, we have identities
\begin{eqnarray*} L_{x}=-J(d\mu(x))^{*}, L_{x}^{*}=d\mu(x)J, \end{eqnarray*}
where $L_{x}^{*}$ is the adjoint operator of $L_{x}$ with respect to the invariant inner product $\langle\cdot,\cdot\rangle$ and the Riemannian metric $g(\cdot,\cdot)=\omega(\cdot,J\cdot)$.
\end{lemma}
\begin{proof} First by the definition of moment map and $L_{x}$, we have
\begin{eqnarray} d\langle\mu(x),\xi\rangle=l(X_{\xi})\omega(x)=\omega(L_{x}\xi,\cdot), \forall\xi\in\mathfrak{g}. \end{eqnarray}

By the definition of the Riemannian metric $g(\cdot,\cdot)$, we have
\begin{eqnarray} \omega(L_{x}\xi,\eta)=\omega(L_{x}\xi,-J(J\eta))=g(L_{x}\xi,-J\eta)=-g(J\eta,L_{x}\xi), \forall \eta\in T_{x}M. \end{eqnarray}
At the same time, $ \forall \eta\in T_{x}M $ we have
\begin{eqnarray} d\langle\mu(x),\xi\rangle(\eta)=g(d\mu(x)\eta,\xi)=g(\eta,(d\mu(x))^{*}\xi)=g(J\eta,J(d\mu(x))^{*}\xi). \end{eqnarray}
By (9),(10),(11), we obtain $L_{x}=-J(d\mu(x))^{*}$. Note that $(L_{x}^{*})^{*}=L_{x}, J^{*}=-J$, so $L_{x}^{*}=d\mu(x)J $. \end{proof}

Let $\mathcal{A}$ be the space of connections on $P$, $\mathcal{A}^{1,1}\subset \mathcal{A}$ be the space of connections whose curvature belong to $\Omega^{1,1}(P\times_{Ad}\mathfrak{g})$. Let $\mathcal{F}=P\times_{G}M\rightarrow X$ be the associated bundle on $X$ with fiber $M$, $\mathcal{L}$ be the space $\Gamma(\mathcal{F})$ of smooth sections of $\mathcal{F}$. \par
In order to define Yang-Mills-Higgs functional for pairs $\mathcal{A}^{1,1}\times \mathcal{L}$, it will be necessary to extend the definition of covariant derivative on vector bundles to general fiber bundles. \par
According to [2], there is a representation of the sections of the associated fibration $\mathcal{F}=P\times_{G}M\rightarrow X$ as functions on the corresponding principle bundle $\pi: P\rightarrow X$. This representation is extremely useful in doing calculations in a coordinate free way.
\begin{theorem}([2])
Let $C_{G}^{\infty}(P,M)$ denote the space of equivariant maps from $P$ to $M$, that is, those maps $u: P\rightarrow M$ that satisfy $u(p\cdot g)=g^{-1}u(p)$. There is a natural isomorphism between $\mathcal{L}=\Gamma(\mathcal{F})$ and $C_{G}^{\infty}(P,M)$, given by sending $u\in C_{G}^{\infty}(P,M)$ to $\tilde{u}$ defined by
\begin{eqnarray*} \tilde{u}(x)=[p,u(p)],  \end{eqnarray*}
here $x\in X$, $p$ is any element of $\pi^{-1}(x)$, and $[p, u(p)]$ is element of $\mathcal{F}=P\times_{G}M$ corresponding to $(p, u(p))\in P\times M$.
\end{theorem}

By Theorem 2.2, from now on we can identify $\mathcal{L}$ as $C_{G}^{\infty}(P,M)$. We define the covariant derivative on fiber bundles $\mathcal{F}=P\times_{G}M\rightarrow X$( see [5] for more details).
$\forall u\in \mathcal{L}=C_{G}^{\infty}(P,M)$. The connection $A$ on $P$ determines a connection on the fiber bundle $\mathcal{F}=P\times_{G}M\rightarrow X$. More precisely, the tangent space of $\mathcal{F}$ at $[p,x]$ is the quotient
\begin{eqnarray*} T_{[p,x]}\mathcal{F}=\frac{T_{p}P\times T_{x}M}{\{(p\xi, -X_{\xi}(x))|\xi\in\mathfrak{ g}\}}, \end{eqnarray*}
the vertical space consists of equivalence classes of the form $[0,w]$ with $w\in T_{x}M$, and the horizontal space consists of those equivalence classes $[v,w]$ where $v\in T_{p}P$ and $w\in T_{x}M$ satisfy $w+X_{A_{p}(v)}(x)=0$. The $covariant$ $derivative$ of a section $u: P\rightarrow M$ with respect to the connection $A$ is the form $d_{A}u: TP\rightarrow u^{*}TM$ given by
\begin{eqnarray*} d_{A}u(p)v=du(p)v+X_{A_{p}(v)}(u(p)). \end{eqnarray*}
It is easy to check that this 1-form is actually $G$-equivariant, and it satisfies $d_{A}u(p)p\xi=0, \forall \xi\in \mathfrak{g}$. Hence it actually defines a covariant derivative on the fibration $\mathcal{F}=P\times_{G}M\rightarrow X$. Recall the definition of the operator $L$, then we can rewrite the covariant derivative as
\begin{eqnarray} d_{A}u=du+L_{u}A, \forall u\in \mathcal{L}=\Gamma(\mathcal{F})=C_{G}^{\infty}(P,M). \end{eqnarray}
We can rewrite the identities of lemma 2.1 as follows
\begin{lemma} $\forall u\in C_{G}^{\infty}(P,M)$, we have identities
\begin{eqnarray} L_{u}=-J(d\mu(u))^{*}, L_{u}^{*}=d\mu(u)J. \end{eqnarray}
\end{lemma}
Now fix a central element $c\in \mathfrak{g}$. The $Yang$-$Mills$-$Higgs$ functional $\mathcal{YMH}: \mathcal{A}^{1,1}\times \mathcal{L}\rightarrow \mathbb{R}$ is defined as
\begin{eqnarray*}\mathcal{YMH}(A,u)=||F_{A}||_{L^{2}}^{2}+||d_{A}u||_{L^{2}}^{2}+||\mu(u)-c||_{L^{2}}^{2}. \end{eqnarray*}
We can see that the functional is composed of the Yang-Mills functional $E_{1}=||F_{A}||_{L^{2}}^{2}$, the energy functional $E_{2}=||d_{A}u||_{L^{2}}^{2}$ and the Higgs functional $E_{3}=||\mu(u)-c||_{L^{2}}^{2}$.
Next we compute its critical points equations.
\subsection{The Euler-Lagrange equations}
$\hspace{15pt}$Adopting [20]'s method, first fix the connection $A$ and take a variation $u_{t}$ of $u$, i.e., $u_{t}\in C_{G}^{\infty}(P,M), u_{0}=u$. So the Yang-Mills functional $E_{1}(A, u_{t})$ is fixed, and we only need calculate $E_{2}$ and $E_{3}$.
Choose a metric $h$ on $P$, $g=g(\cdot,\cdot)=\omega(\cdot,J\cdot)$ on $M$ and their corresponding Levi-civita connections $\nabla,\bar{\nabla}$. Further $\nabla,\bar{\nabla}$ induce connections $\nabla^{*}$ on $T^{*}P$, $\bar{\nabla}$ on $u^{*}TM\rightarrow P$, $\nabla^{*}\otimes \bar{\nabla}$ on $T^{*}P\otimes u^{*}TM$. Let $U\subset P, (x^{i})\in U$, $u(U)\subset V\subset M, (y^{\alpha})\in V$ be local coordinates. In the local coordinates, we can write
\begin{eqnarray*} d_{A}(u)=u^{i}_{\beta}dx^{i}\otimes (\frac{\partial}{\partial y^{\beta}}\cdot u), \end{eqnarray*}
where \begin{eqnarray*}
u^{i}_{\beta}&=&d_{A}u(\frac{\partial}{\partial x^{i}}\otimes dy^{\beta})\\
&=&(du+X_{A}u)(\frac{\partial}{\partial x^{i}}\otimes dy^{\beta})\\
&=&\frac{\partial u^{\beta}}{\partial x^{i}}+X_{A(\frac{\partial}{\partial x^{i}})}u^{\beta}\\
&=&\frac{\partial u^{\beta}}{\partial x^{i}}+X_{A_{i}}u^{\beta},
\end{eqnarray*}
and
\begin{eqnarray*}
<dx^{i}\otimes \frac{\partial}{\partial y^{\alpha}}, dx^{j}\otimes \frac{\partial}{\partial y^{\beta}}>&=&h^{ij}g_{\alpha\beta},
\end{eqnarray*}
so
\begin{eqnarray*}
<d_{A}u, d_{A}u>&=&<(\frac{\partial u^{\beta}}{\partial x^{i}}+X_{A_{i}}u^{\beta})dx^{i}\otimes \frac{\partial}{\partial y^{\beta}},(\frac{\partial u^{\alpha}}{\partial x^{j}}+X_{A_{j}}u^{\alpha})dx^{j}\otimes \frac{\partial}{\partial y^{\alpha}}>\\
&=&h^{ij}g_{\alpha\beta}(\frac{\partial u^{\beta}}{\partial x^{i}}+X_{A_{i}}u^{\beta})(\frac{\partial u^{\alpha}}{\partial x^{j}}+X_{A_{j}}u^{\alpha}).
\end{eqnarray*}
Therefore we have \begin{eqnarray*}
E_{2}(A, u_{t})=\int_{X}<d_{A}u_{t},d_{A}u_{t}>=\int_{X}h^{ij}g_{\alpha\beta}(u_{t})(\frac{\partial u_{t}^{\beta}}{\partial x^{i}}+X_{A_{i}}u_{t}^{\beta})(\frac{\partial u_{t}^{\alpha}}{\partial x^{j}}+X_{A_{j}}u_{t}^{\alpha}). \end{eqnarray*}
By compatibility between metrics and connections above,
\begin{eqnarray*} \nabla_{t}(h^{ij}g_{\alpha\beta})=0, \nabla_{k}(h^{ij}g_{\alpha\beta})=0. \end{eqnarray*}
By symmetry, \begin{eqnarray*}
\frac{d}{dt}|_{t=0}E_{2}(A, u_{t})=2\int_{X} h^{ij}g_{\alpha\beta}(u_{t})\nabla_{t}|_{t=0}(\frac{\partial u_{t}^{\beta}}{\partial x^{i}}+X_{A_{i}}u_{t}^{\beta})(\frac{\partial u_{t}^{\alpha}}{\partial x^{j}}+X_{A_{j}}u_{t}^{\alpha}). \end{eqnarray*}
Now we need compute $\nabla_{t}|_{t=0}(\frac{\partial u_{t}^{\beta}}{\partial x^{i}}+X_{A_{i}}u_{t}^{\beta})$.
Since $\nabla_{t}(\frac{\partial u_{t}^{\beta}}{\partial x^{i}})=\nabla_{i}(\frac{\partial u^{\beta}_{t}}{\partial t})$, we only need compute $\nabla_{t}(X_{A_{i}}(u^{\beta}_{t}))$.
Using the equivalent definition of connection by parallel translation, we get \begin{eqnarray*}
\nabla_{t}(X_{A_{i}}(u^{\beta}_{t}))=\bar\nabla_{V^{\beta}}(X_{A_{i}}). \end{eqnarray*}
Let $V^{\beta}=\frac{\partial u^{\beta}_{t}}{\partial t}|_{t=0}$, then
\begin{eqnarray*}
\frac{d}{dt}|_{t=0}E_{2}(A, u_{t})
&=&2\int_{X} h^{ij}g_{\alpha\beta}(u)(\nabla_{i}V^{\beta}+\bar\nabla_{V}(A_{i})(\frac{\partial}{\partial y^{\beta}}))(\frac{\partial u^{\alpha}}{x^{j}}+X_{A_{j}}(u^{\alpha}))\\
&=&2\int_{X}<\nabla V, d_{A}u>+<\nabla_{V}A, d_{A}u>\\
&=&2\int_{X}<\nabla V+\nabla_{V}A, d_{A}u>\\
&=&2\int_{X}<\nabla_{A}V, d_{A}u>,\\
\end{eqnarray*}
where $A\in \Omega^{1}(P, \mathfrak{g}), V\in \Gamma(u^{*}TM), X_{A},  \nabla V \in \Omega^{1}(P, u^{*}TM)$.\\

As for $E_{3}(A, u_{t})=||\mu(u_{t})-c||_{L^{2}}^{2}=\int_{X}|\mu(u_{t})-c|^{2}dvol$, we have \begin{eqnarray*}
\frac{d}{dt}|_{t=0}E_{3}&=&2\int<d\mu(\frac{d u_{t}}{dt}|_{t=0}), \mu(u_{0})-c>\\
&=&2\int <d\mu(V), \mu(u)-c>dvol. \end{eqnarray*}
Thus the first variation formula of Yang-Mills-Higgs functional for $u$ is:\begin{eqnarray*}
\frac{d}{dt}|_{t=0}\mathcal{YMH}(A, u_{t})&=&\frac{d}{dt}|_{t=0}(E_{2}+ E_{3})\\
&=&2\int_{X}(<\nabla_{A}V, d_{A}u>+<d\mu(V), \mu(u)-c>)\\
&=&2\int_{X}(<V, \nabla^{*}_{A}d_{A}u>+<V, (d\mu)^{*}(\mu(u)-c)>)\\
&=&2\int_{X}<V, \nabla^{*}_{A}d_{A}u+(d\mu)^{*}(\mu(u)-c)>),
\end{eqnarray*}
where $V=\frac{d u_{t}}{dt}|_{t=0}\in \Gamma( u^{*}TM)$ is the variation vector field, $\nabla_{A}: \Gamma(u^{*}TM)\rightarrow \Omega^{1}(P, u^{*}TM)$ the covariant derivative, and $\nabla^{*}_{A}$ is the adjoint operator of $\nabla_{A}.$\\

In all, we get the first Euler-Lagrange equation with respect to $u$:
\begin{eqnarray*} \frac{d}{dt}|_{t=0}\mathcal{YMH}(A, u_{t})=0\Leftrightarrow \nabla^{*}_{A}d_{A}u+(d\mu)^{*}(\mu(u)-c)=0.\\\end{eqnarray*}

Next we fix $u$ and study the first variation formula for $A\in \mathcal{A}$. It is obvious that we only need compute $E_{1}=||F_{A}||_{L^{2}}^{2}, E_{2}=||d_{A}u||_{L^{2}}^{2}$. \par
First we recall the variation of connection ([10]). Because space of connection on principle G-bundle $P\rightarrow X$ is an affine space: $d+\Omega^{1}(X, ad P)$, where $ad P$ is the associate bundle of $P$ with respect to the adjoint representation of $G$. Thus it suffices to vary $A$ along lines $A_{t}=A+tB, A_{0}=A, \forall B\in d+\Omega^{1}(X, ad P)$. It is easy to compute \begin{eqnarray*}
F_{A_{t}}&=&F_{A}+t D_{A}B+t^{2}B\wedge B,\\
 d_{A_{t}}u&=&du+X_{A+tB}(u)=d_{A}u+tX_{B}(u).
 \end{eqnarray*}
So we have
 \begin{eqnarray*}
E_{1}+E_{2}&=&\int_{X}<F_{A}+t D_{A}B+t^{2}B\wedge B,F_{A}+t D_{A}B+t^{2}B\wedge B>\\
&&+<d_{A}u+tX_{B}(u), d_{A}u+tX_{B}(u)>\\
&=&2t\int_{X}(<d_{A}(u), X_{B}(u)>+<F_{A}, D_{A}B>)+\cdots \end{eqnarray*}
Therefore
  \begin{eqnarray*}
\frac{d \mathcal{YMH}(A_{t}, u)}{dt}|_{t=0}=0
&& \Leftrightarrow  \int_{X}<d_{A}u, X_{B}(u)>+<F_{A}, D_{A}B>\\
&=&\int_{X}<d_{A}u, L_{u}B>+<F_{A}, D_{A}B>\\
&=&\int_{X}<L^{*}_{u}d_{A}u+D^{*}_{A}F_{A}, B>=0, \end{eqnarray*}
where $L_{u}: \mathfrak{g}\rightarrow u^{*}TM$ defined by $\xi\in \mathfrak{g}\mapsto X_{\xi}(u)$, $L^{*}_{u}: u^{*}TM\rightarrow \mathfrak{g}$ is adjoint of $L_{u}$, $D^{*}_{A}$ is adjoint operator of $D_{A}$. \par
Therefore we obtain the second Euler-Lagrange equation of the Yang-Mills-Higgs functional for $A$:\begin{eqnarray*}
L^{*}_{u}d_{A}u+D^{*}_{A}F_{A}=0. \end{eqnarray*}
To sum up, we obtain
\begin{theorem} The Euler-Lagrange equations of Yang-Mills-Higgs functional \begin{eqnarray*}\mathcal{YMH}(A,u)=||F_{A}||_{L^{2}}^{2}+||d_{A}u||_{L^{2}}^{2}+||\mu(u)-c||_{L^{2}}^{2} \end{eqnarray*} are as follows: \begin{eqnarray*}
L^{*}_{u}d_{A}u+D^{*}_{A}F_{A}&=&0;\\ \nabla^{*}_{A}d_{A}u+(d\mu)^{*}(\mu(u)-c)&=&0.\\ \end{eqnarray*}
\end{theorem}

$\hspace{-17pt}$$\textbf{Remark 2.1.}$  When $G={1}$, then the Euler-Lagrange equations above can be reduced to the equation of harmonic maps \begin{eqnarray*}
\nabla^{*}du=-Trace \nabla du=0. \end{eqnarray*}.\\
  When the fiber $M={pt}$, then  the Euler-Lagrange equations above can be reduced to the equation of Yang-Mills connections [1, 6]:
\begin{eqnarray*}
D^{*}_{A}F_{A}=0. \end{eqnarray*}
Therefore, the Euler-Lagrange equations of Yang-Mills-Higgs functional are common generalization of harmonic maps and Yang-Mills connections.\par
In addition, there is the gauge group $\mathcal{G}=Aut(P)$ action on pair $(A, u)$. Similar to the Yang-Mills functional , the Yang-Mills-Higgs functional is invariant under group $\mathcal{G}$ action, therefore we have \\
\begin{proposition} The Euler-Lagrange equations above are invariant under the gauge group $\mathcal{G}=Aut(P)$ action. \\
\end{proposition}
\section{Heat flow}
$\hspace{15pt}$In this section, we will prove the long term existence of the heat flow of Yang-Mills-Higgs functional (2). The heat flow equations are \begin{eqnarray*}
\frac{\partial A}{\partial t}&=&-L^{*}_{u}d_{A}u-D^{*}_{A}F_{A};\\
\frac{\partial u}{\partial t}&=&-\nabla^{*}_{A}d_{A}u-(d\mu)^{*}(\mu(u)-c).
\end{eqnarray*}
Modifying Hong [8]'s method, we first prove the energy inequality of Yang-Mills-Higgs functional.
\subsection{Energy inequality}
$\hspace{15pt}$Consider the behavior of the Yang-Mills-Higgs functional  $\mathcal{YMH}(A_{t}, u_{t})$ along the heat flow equations  \begin{eqnarray*}
\frac{\partial A}{\partial t}&=&-L^{*}_{u}d_{A}u-D^{*}_{A}F_{A};\\
\frac{\partial u}{\partial t}&=&-\nabla^{*}_{A}d_{A}u-(d\mu)^{*}(\mu(u)-c).
\end{eqnarray*}
Following [8], we can similarly obtain
\begin{lemma}(Energy inequality) If the pair $(A_{t}, u_{t})$ is a solution to Yang-Mills-Higgs heat flow (5),(6) on $X\times [0,\infty)$, then for $t<\infty$
\begin{equation} \mathcal{YMH}(A_{t}, u_{t})+2\int^{t}_{0}\int_{X}(|\frac{\partial u}{\partial \tau}|^{2}+|L^{*}_{u}d_{A}u+D^{*}_{A}F_{A}|^{2})dxd\tau= \mathcal{YMH}(A_{0}, u_{0}). \end{equation}
\end{lemma}
\begin{proof} On one side, by (5), (6) we have \begin{eqnarray*}
\int_{X}|\frac{\partial u}{\partial t}|^{2}dx
&=&(\frac{\partial u}{\partial t},-\nabla^{*}_{A}d_{A}u-(d\mu)^{*}(\mu(u)-c))((6))\\
&=&-(\frac{\partial u}{\partial t},\nabla^{*}_{A}d_{A}u)-(\frac{\partial u}{\partial t},(d\mu)^{*}(\mu(u)-c))\\
&=&-(\nabla_{A}\frac{\partial u}{\partial t}, d_{A}u))-(\frac{\partial u}{\partial t}, (d\mu)^{*}(\mu(u)-c))\\
&=&-(\frac{\partial (d_{A}u)}{\partial t}, d_{A}u))+(\frac{\partial A}{\partial t}(u), d_{A}u)-(\frac{\partial}{\partial t}, (d\mu)^{*}(\mu(u)-c))\\
&=&-(\frac{\partial (d_{A}u)}{\partial t}, d_{A}u))-(L^{*}_{u}d_{A}u+D^{*}_{A}F_{A})(u), d_{A}u)\\
&&-((d\mu)(\frac{\partial}{\partial t}), \mu(u)-c)((5))\\
&=&-\frac{1}{2}\frac{d}{d t}\int_{X}(|d_{A}u|^{2}+|\mu(u)-c|^{2})-(L^{*}_{u}d_{A}u+D^{*}_{A}F_{A}), L^{*}_{u}d_{A}u).
\end{eqnarray*}
On the other side,
\begin{eqnarray*}
\frac{1}{2}\frac{d}{d t}\int_{X} |F_{A}|^{2}&=&(\frac{\partial F_{A}}{\partial t}, F_{A})\\
&=&(D_{A}\frac{\partial A}{\partial t}, F_{A})\\
&=&-(D_{A}(L^{*}_{u}d_{A}u+D^{*}_{A}F_{A}), F_{A})((5))\\
&=&-(L^{*}_{u}d_{A}u+D^{*}_{A}F_{A}, D^{*}_{A}F_{A}).
\end{eqnarray*}
Therefore, we have \begin{eqnarray*}
\frac{1}{2}\frac{d}{d t}\int_{X}(|F_{A}|^{2}+|d_{A}u|^{2}+|\mu(u)-c|^{2})dx=-\int_{X}(|\frac{\partial u}{\partial t}|^{2}+|L^{*}_{u}d_{A}u+D^{*}_{A}F_{A}|^{2})dx. \end{eqnarray*}
Integration on both sides we complete the proof. \end{proof}

\begin{corollary} Yang-Mills-Higgs functional decreases along the heat flow.  \end{corollary}

\subsection{Local existence}
$\hspace{15pt}$Following Donaldson's approach [6, 8], consider the complex gauge group ${\mathcal{G}}^{\mathbb{C}}$, which acts on $\mathcal{A}^{1,1}$ by
\begin{eqnarray} \bar{\partial}_{g(A)}=g\cdot \bar{\partial}_{A}\cdot g^{-1},  {\partial}_{g(A)}={g^{*}}^{-1}\cdot {\partial}_{A}\cdot g^{*} \end{eqnarray}
where $g^{*}$ denotes the conjugate transpose of $g$. Extending the action of the unitary gauge group,
\begin{eqnarray*}  \mathcal{G}=\{g\in {\mathcal{G}}^{\mathbb{C}}|h(g)=g^{*}g=I\},\end{eqnarray*}
which means
\begin{eqnarray} g^{-1}\cdot d_{g(A)}\cdot g=\bar{\partial}_{A}+h^{-1}{\partial}_{A}h,
g^{-1}F_{g(A)}g=F_{A}+\bar{\partial}_{A}(h^{-1}{\partial}_{A}h), \end{eqnarray}
where $h=g^{*}g$.\par
Consider following heat equation for a one-parameter family $H_{t}$ of metrics on some holomorphic bundle over $X$:\begin{equation}
\frac{\partial H}{\partial t}=-2iH[\hat{F_{t}}+\mu(u)-c], \end{equation}
where $\hat{F_{t}}=\wedge F_{H_{t}}$.
This heat equation is completely equivalent to the equation:
 \begin{equation}
\frac{\partial h}{\partial t}=-\triangle_{A_{0}}h-i[h\hat{F}_{A_{0}}+\hat{F}_{A_{0}}h+2(\mu(u_{0})-c)h]+2i\wedge({\bar \partial}_{A_{0}}h h^{-1}{\partial}_{A_{0}}h), \end{equation}
with $h(0)=I$, where $\triangle_{A_{0}}$ is the Laplacian defined through an integrable connection $A_{0} ({\bar \partial}^{2}_{A_{0}}=0)$, h is a positive self-adjoint endomorphism of a unitary bundle such that $H(t)=H_{0}h(t)$, and $H_{0}$ is the initial Hermitian metric on $P$.
As pointed out in [6, 8], (18) is a nonlinear parabolic equation, so we obtain a short time solution to (18) by standard parabolic PDE theory. Therefore we obtain a short-time solution to (18).\par

Then we need do some necessary calculations.
First we need generalize the K$\ddot{\textmd{a}}$hler identities ([6, 7, 11]) from vector bundle-values case to fiber bundle-values case where the fibre is a manifold not a vector space.\par
Fix the K$\ddot{\textmd{a}}$hler metric on $X$ as
\begin{eqnarray*} h_{0}=h_{k\bar{j}}dz_{k}\otimes d\bar{z_{j}}, \end{eqnarray*}
and define the K$\ddot{\textmd{a}}$hler form $\omega_{X}$ as
\begin{eqnarray*} \omega_{X}=\frac{i}{2}h_{k\bar{j}}dz_{k}\wedge d\bar{z_{j}}, i=\sqrt{-1}. \end{eqnarray*}
Let $A$ be a connection on the principle bundle $P$ over the K$\ddot{\textmd{a}}$hler manifold $X$, so we have
\begin{eqnarray*}  d_{A}: C_{G}^{\infty}(P,M)\rightarrow \Omega_{G}^{1}(P, u^{*}TM),
d_{A}u:=du+L_{u}A.
\end{eqnarray*}
Then we have
\begin{eqnarray*} {\partial}_{A}:C_{G}^{\infty}(P,M)\rightarrow \Omega_{G}^{1,0}(P, u^{*}TM)\\
{\partial}_{A}u=\frac{1}{2}(d_{A}u-J\cdot d_{A}u\cdot j),
\end{eqnarray*}
and
\begin{eqnarray*}
{\bar \partial}_{A}:C_{G}^{\infty}(P,M)\rightarrow \Omega_{G}^{0,1}(P, u^{*}TM)\\
{\bar \partial}_{A}u=\frac{1}{2}(d_{A}u+J\cdot d_{A}u\cdot j),
\end{eqnarray*}
making up  \begin{eqnarray*}d_{A}={\partial}_{A}+{\bar \partial}_{A}, \end{eqnarray*}
where $j$ is the fixed complex structure of the K$\ddot{\textmd{a}}$hler manifold $X$, and $J$ is a $\omega$-compatible almost complex structure of the symplectic manifold $M$.\par

 We denote $\Omega^{k}$ the $k$-form and recall the decomposition
\begin{eqnarray*}  \Omega^{k}=\sum_{p+q=k}\Omega^{p,q}, \end{eqnarray*}
of $k$-forms into $(p,q)$ forms. On the K$\ddot{\textmd{a}}$hler manifold $X$ with K$\ddot{\textmd{a}}$hler form $\omega_{X}$ we define
\begin{eqnarray*} L: \Omega^{p,q}\rightarrow \Omega^{p+1,q+1}, L(\eta)=\eta\wedge\omega_{X}. \end{eqnarray*}
The we can define an algebraic trace operator $\wedge$ by
\begin{eqnarray*} \wedge=L^{*}:\Omega^{p,q}\rightarrow \Omega^{p-1,q-1}. \end{eqnarray*}
Through the definition of $\wedge$, we have

\begin{lemma}(Generalized K$\ddot{\textmd{a}}$hler identities) When the fiber $M$ of fiber bundle $\mathcal{F}=P\times_{G}M\rightarrow X$ is a compact $K\ddot{\textmd{a}}hler$ manifold, or equivalently $J$ is a complex structure of $M$, we have
\begin{eqnarray}
{\bar \partial}^{*}_{A}=i[\partial_{A}, \wedge];\\
{\partial}^{*}_{A}=-i[{\bar \partial}_{A}, \wedge].
\end{eqnarray}
\end{lemma}
\begin{proof} Our case is essentially same as vector bundle-value case [6, 7, 11]. \end{proof}

Then we will do some calculations.
\begin{lemma} Let $A_{0}\in \mathcal{A}^{1,1}$ is a connection on $P$. Assume that $u_{0}\in C_{G}^{\infty}(P,M)$ is holomorphic with respect to the connection $A_{0}$ and $J$ ($A_{0}$ and $J$ determine a holomorphic structure on $P\times_{G}M$), i.e., ${\bar \partial}_{A_{0}}u_{0}=0$. Assume that $A(t)=g(t)(A_{0}), u(t)=g(t)(u_{0}), g(t)\in \mathcal{G}^{\mathbb{C}}$. Then we have \begin{eqnarray}
{\bar \partial}_{A(t)}u(t)=0, F_{A(t)}\in \Omega^{1,1}.
\end{eqnarray}
\end{lemma}
\begin{proof} By (15) and (16), we have \begin{eqnarray*}
{\bar \partial}_{g(A_{0})}&=&g\cdot {\bar \partial}_{A_{0}}\cdot g^{-1}; \\
g^{-1}F_{g(A_{0})}g&=& F_{A_{0}}+{\bar \partial}_{A_{0}}(h^{-1}\partial h), \end{eqnarray*}
where $h=g^{*}g$.
Thus by the transformation formulas and conditions above we have \begin{eqnarray*}
F_{A(t)}=F_{g^{*}(A_{0})}=g(F_{A_{0}}+{\bar \partial}_{A_{0}}(h^{-1}\partial h))g^{-1}\in \Omega^{1,1}.\end{eqnarray*}
And
\begin{eqnarray*} {\bar \partial}_{A(t)}u(t)=g\cdot {\bar \partial}_{A_{0}}\cdot g^{-1}gu_{0}=g{\bar \partial}_{A_{0}}u_{0}=0. \end{eqnarray*}
This completes the proof. \end{proof}

\begin{lemma} Assume the same conditions as lemma 3.4. and write\begin{eqnarray*} A=A(t)=g(t)(A_{0}), u=u(t)=g(t)(u_{0}),  \end{eqnarray*} we have \begin{eqnarray}
({\bar \partial}_{A}-\partial_{A})(\mu(u))= i\cdot L^{*}_{u}(d_{A}u). \end{eqnarray}
\end{lemma}
\begin{proof} By previous definitions of ${\bar \partial}_{A}, {\partial}_{A}$,
we have \begin{equation} ({\bar \partial}_{A}-\partial_{A})(\mu(u))=i\cdot d_{A}(\mu(u))\cdot j=i\cdot d\mu(u)(d_{A}u\cdot j), \end{equation}
where $j$ is the fixed complex structure of $X, i=\sqrt{-1}$.\par On the other hand, by lemma 2.3, we have  \begin{equation} L^{*}_{u}(d_{A}u)=d\mu(u)J(d_{A}u)=d\mu(u)(J\cdot d_{A}u). \end{equation}\\ By lemma 3.4, we know \begin{eqnarray*}
0={\bar \partial}_{A(t)}u(t)&=&{\bar \partial}_{A}u=d_{A}u+J\cdot d_{A}u\cdot j \end{eqnarray*}
\begin{eqnarray}\Leftrightarrow J\cdot d_{A}u&=&d_{A}u\cdot j. \end{eqnarray}
Thus combining (23),(24) and (25) we complete the proof.\end{proof}
\begin{lemma} Suppose that $M$ is a compact K$\ddot{\textmd{a}}hler$ manifold. Let $A_{0}\in \mathcal{A}^{1,1}$, and $u_{0}\in C_{G}^{\infty}(P,M)$ such that ${\bar \partial}_{A_{0}}u_{0}=0$. Assume that $A=A(t)=g(t)(A_{0}), u=u(t)=g(t)(u_{0}), g(t)\in \mathcal{G}^{\mathbb{C}}$. We have
\begin{eqnarray}d_{A}^{*}F_{A}=i(\partial_{A}-{\bar \partial}_{A})\wedge F_{A}. \end{eqnarray}
\end{lemma}
\begin{proof} First by Bianchi identity $d_{A}F_{A}=0$ we have \begin{equation}{\partial}_{A}F_{A}=-{\bar \partial}_{A}F_{A}. \end{equation}
By lemma 3.4 we know $F_{A}\in \Omega^{1,1}$, so $ {\partial}_{A}F_{A}\in \Omega^{2,1}, {\bar \partial}_{A}F_{A}\in \Omega^{1,2}$, we have \begin{equation} {\partial}_{A}F_{A}={\bar \partial}_{A}F_{A}=0. \end{equation}
Thus by generalized K$\ddot{\textmd{a}}$hler identities (lemma 3.3) and (28) we have
\begin{eqnarray*} d_{A}^{*}F_{A}&=&({\bar \partial}_{A}^{*}+{\partial}_{A}^{*})F_{A}\\
&=&i({\partial}_{A}-{\bar \partial}_{A})\wedge F_{A}-i\wedge({\partial}_{A}F_{A}-{\bar \partial}_{A}F_{A})\\
&=&i({\partial}_{A}-{\bar \partial}_{A})\wedge F_{A}. \end{eqnarray*}
This completes the proof. \end{proof}
\begin{lemma} Assume the same conditions as lemma 3.6, we have \begin{eqnarray*}
d_{A}^{*}d_{A}=i\wedge({\bar \partial}_{A}{\partial}_{A}-{\partial}_{A}{\bar \partial}_{A})
\end{eqnarray*} and
\begin{eqnarray*}
d_{A}^{*}d_{A}-i\wedge F_{A}=2{\bar \partial}_{A}^{*}{\bar \partial}_{A}.
\end{eqnarray*}
\end{lemma}
\begin{proof} By generalized K$\ddot{\textmd{a}}$hler identities and lemma 3.4, \begin{eqnarray*}
d^{*}_{A}d_{A}&=&({\bar \partial}^{*}_{A}+{\partial}^{*}_{A})({\bar \partial}_{A}+{\partial}_{A})\\
&=&i\wedge({\bar \partial}_{A}{\partial}_{A}-{\partial}_{A}{\bar \partial}_{A}),
\end{eqnarray*}
and we also have \begin{eqnarray*} d^{*}_{A}d_{A}&=&{\partial}^{*}_{A}{\partial}_{A}+{\bar \partial}^{*}_{A}{\bar \partial}_{A}. \end{eqnarray*}
Then
\begin{eqnarray*}
{\partial}^{*}_{A}{\partial}_{A}-{\bar \partial}^{*}_{A}{\bar \partial}_{A}&=&i\wedge{\bar \partial}_{A}{\partial}_{A}+i\wedge{\partial}_{A}{\bar \partial}_{A}\\
&=&i\wedge(d^{2}_{A})\\
&=&i\wedge F_{A}. \end{eqnarray*}
Therefore we have \begin{eqnarray*}d_{A}^{*}d_{A}-i\wedge F_{A}=2{\bar \partial}_{A}^{*}{\bar \partial}_{A}. \end{eqnarray*} This completes the proof of lemma 3.7. \end{proof}

Now take any $g\in \mathcal{G}^{\mathbb{C}}$ such that $g^{*}g=h$ (for example $g=h^{1/2}$). Since $h$ solve (18), we have \begin{equation}
\frac{\partial g}{\partial t}g^{-1}+g^{*}\frac{\partial g^{*}}{\partial t}=-2i[\wedge F_{g(A_{0})}+\mu(u)-c]. \end{equation}
Here we use a fact that the moment map $\mu$ is equivariant, i.e., $g\mu(u_{0})g^{-1}=\mu(g(u_{0}))=\mu(u)$.
Then as [8], at the time $t$,
\begin{eqnarray*} \frac{\partial A}{\partial t}&=&\frac{\partial}{\partial\epsilon}|_{\epsilon=0}(\partial_{(g+\epsilon\partial_{t}g)A_{0}}+{\bar\partial}_{(g+\epsilon\partial_{t}g)A_{0}})\\
&=&\partial_{A}({g^{*}}^{-1}\partial_{t}g^{*})+{\bar\partial}_{A}(g\partial_{t}g^{-1}).\end{eqnarray*}
So we have
\begin{eqnarray*}
\frac{\partial A}{\partial t}&=&{\partial}_{A}(g^{*}\partial_{t}g^{*})-{\bar \partial}_{A}(\partial_{t}g g^{-1})\\
&=&-{ \partial}_{A}i[\wedge F_{A}+\mu(u)-c]-\frac{1}{2}{\partial}_{A}(\partial_{t}g g^{-1})+\frac{1}{2}{\partial}_{A}(g^{*}\partial_{t}g^{*})+\\
& & {\bar \partial}_{A}i[\wedge F_{A}+\mu(u)-c]+\frac{1}{2}{\bar \partial}_{A}(g^{*}\partial_{t}g^{*})-\frac{1}{2}{\bar \partial}_{A}(\partial_{t}g g^{-1}) \\
 & =& i({\bar \partial}_{A}-{\partial}_{A})[\wedge F_{A}+\mu(u)-c]+d_{A}(\alpha(t))\\
& =&-i(\partial_{A}-{\bar \partial}_{A})\wedge F_{A}+i({\bar \partial}_{A}-\partial_{A})(\mu(u)-c)+d_{A}(\alpha(t)).\\
 \end{eqnarray*}
Due to lemma 3.5 and lemma 3.6, we have
 \begin{eqnarray}\frac{\partial A}{\partial t}&=&-d_{A}^{*}F_{A}-L^{*}_{u}(d_{A}u)+d_{A}(\alpha(t)),\end{eqnarray}
 where $\alpha(t)$ is defined by $\alpha(t)=\frac{1}{2}(g^{*}\frac{\partial g^{*}}{\partial t}-\partial_{t}g g^{-1})$.\\

 Using (29) we obtain
 \begin{eqnarray*}
\frac{\partial u}{\partial t}& = & \partial_{t}g g^{-1}u\\
 & =& -i[\wedge F_{A}+\mu(u)-c]u-\frac{1}{2}(g^{*}\frac{\partial g^{*}}{\partial t}-\partial_{t}g g^{-1})u\\
 & =&-i\wedge F_{A}u-iL_{u}(\mu(u)-c)-\alpha(t)u.
\end{eqnarray*}
By lemma 2.3, lemma 3.4 and lemma 3.7, we have
\begin{eqnarray} \frac{\partial u}{\partial t}& =&-d_{A}^{*}d_{A}u-(d\mu)^{*}(\mu(u)-c)-\alpha(t)u. \end{eqnarray}

Assume $h$ is the solution of (18), let $g=h^{\frac{1}{2}}$. The corresponding pair $(\tilde{A}(t), \tilde{u}(t))=(g(A_{0}), g(u_{0}))$ thus is a solution to (30) and (31). Through a gauge transformation of the equivalent flow (30), (31), we prove the local existence of Yang-Mills-Higgs heat flow. \par
Now let $S(t)$ be the unique smooth solution to the following initial value problem:
\begin{eqnarray}
\frac{d S}{dt}\cdot S^{-1}=- \alpha(t), S(0)=I, \end{eqnarray}
where \begin{eqnarray*}
\alpha(t)=\frac{1}{2}(g^{*}\frac{\partial g^{*}}{\partial t}-\partial_{t}g g^{-1}). \end{eqnarray*}
Let \begin{eqnarray*}d_{A}=S^{-1}\cdot d_{\tilde{A}}\cdot S, u=S^{-1}\tilde{u}, \end{eqnarray*}
we have \begin{eqnarray*}
S^{-1}\cdot{d^{*}_{\tilde A}F_{\tilde A}}\cdot S=d^{*}_{A}F_{A}; \\
S^{-1}\cdot L^{*}_{\tilde{u}}(d_{\tilde A}\tilde u)\cdot S= L^{*}_{u}(d_{A}u), \end{eqnarray*}
and \begin{eqnarray*}
d_{\tilde A}(\alpha)=d_{\tilde A}\cdot \alpha-\alpha\cdot d_{\tilde A}.\end{eqnarray*}
Combining these equalities with (32), we have
\begin{eqnarray*}
\frac{\partial A}{\partial t}&=&-d_{A}^{*}F_{A}-L^{*}_{u}(d_{A}u);\\
\frac{\partial u}{\partial t}&=&-d_{A}^{*}d_{A}u-(d\mu)^{*}(\mu(u)-c).
\end{eqnarray*}
We ultimately get a smooth solution on $X\times [0, \epsilon)$. This completes the proof of the following local existence theorem:
\begin{theorem}[Local existence] Assume that both the base manifold $X$ and the fiber $M$ are compact K$\ddot{\textmd{a}}$hler manifolds, where $M$ supports a Hamiltonian $G$-group action, and $P$ is a holomorphic G-principal bundle over $X$. Let $A_{0}\in \mathcal{A}^{1,1}$ be a given smooth connection on $P$ with curvature $F_{A_{0}}$ of type (1,1) and let $u_{0}\in\mathcal{L}$ be holomorphic, i.e. ${\bar \partial}_{A_{0}}u_{0}=0$. Then there exist a positive constant $\epsilon >0$ and a smooth solution $(A(x,t), u(x,t))$ such that $(A, u)$ solve the Yang-Mills-Higgs heat flow (5),(6) in $X \times [0, \epsilon)$ with initial values $A(x, 0)=A_{0}$ and $u(x, 0)=u_{0}$.
\end{theorem}

\subsection{Global existence}
$\hspace{15pt}$Recall the equivalent heat equation (17)
\begin{eqnarray*}
\frac{\partial H_{t}}{\partial t}=-2iH_{t}[\hat{F_{t}}+\mu(u)-c]. \end{eqnarray*}
In order to study the uniqueness of heat flow, we also introduce a distance between metrics as in [6].\\

$\hspace{-17pt}$$\textbf{Definition 3.1.}$ For any two metrics $H, K$ on the principle $G$-bundle $P$, set \begin{eqnarray*}
\tau(H,K)&=&Tr(H^{-1}K);\\
\sigma(H,K)&=&\tau(H,K)+\tau(K,H)-2rank(G). \end{eqnarray*}
It follows immediately from that the inequality \begin{eqnarray}
\lambda+\lambda^{-1}\geq2, \forall \lambda\geq0, \end{eqnarray}
that $\sigma(H,K)\geq 0$ with equality if and only if $H=K$.
Following [6], we can prove a lemma which is exactly the same as [6, Prop. 12]
\begin{lemma}If $H_{t}, K_{t}$ are two solutions of the evolution equation (17), then writing $\sigma=\sigma(H,K)$, we have \begin{eqnarray}
\frac{\partial \sigma}{\partial t}+\Delta \sigma \leq 0. \end{eqnarray}
\end{lemma}
Using lemma 3.9, we can prove local uniqueness of the heat equation (17)
\begin{theorem}(Local uniqueness) Any two smooth solutions $H_{t}, K_{t}$, which are defined for $0\leq t< \epsilon$, are continuous at $t=0$, and we have the same initial condition $H_{0}=K_{0}$, agree for all $t\in [0, \epsilon)$.
\end{theorem}
\begin{proof} Using lemma 3.9, apply the maximum principle to $\sigma(H_{t}, K_{t})$. \end{proof}

Now we prepare to prove global existence of the heat flow (5),(6). Recall that we take an invariant inner product on $\mathfrak{g}$, which induces a norm $|\cdot|$.  Then set \begin{equation}
\hat{e}=|\wedge F_{A}+\mu(u)-c|^{2}. \end{equation}
By Theorem 3.8, assume that $(A,u)$ is a solution of the Yang-Mills-Higgs heat flow (5),(6) on $X\times[0, T)$:
\begin{eqnarray*}
\frac{\partial A}{\partial t}&=&-L^{*}_{u}d_{A}u-D^{*}_{A}F_{A};\\
\frac{\partial u}{\partial t}&=&-\nabla^{*}_{A}d_{A}u-(d\mu)^{*}(\mu(u)-c),
\end{eqnarray*}
with curvature $F_{A}\in \Omega^{1,1}$ and ${\bar \partial}_{A}u=0$ on $X\times [0, T)$, then we have
\begin{theorem} Assume that both the base manifold $X$ and the fiber $M$ are compact K$\ddot{\textmd{a}}$hler manifolds, where $M$ supports a Hamiltonian $G$-group action, and $P$ is a holomorphic G-principal bundle over $X$. Let $A_{0}\in \mathcal{A}^{1,1}$ be a given smooth connection on $P$ with curvature $F_{A_{0}}$ of type (1,1) and let $u_{0}\in\mathcal{L}$ be holomorphic, i.e. ${\bar \partial}_{A_{0}}u_{0}=0$. Suppose that $(A, u)$ is local solution of Yang-Mills-Higgs heat flow (5),(6) in $X \times [0, \epsilon)$ with initial values $A_{0}$ and $u_{0}$, then we have
\begin{equation} (\frac{\partial}{\partial t}+\triangle)\hat{e}\leq 0, \end{equation}
where $\triangle=\nabla^{*}_{A}\nabla_{A}$ is the Laplacian.
\end{theorem}
To prove this theorem, it suffices to prove an equality
\begin{lemma} Assume the same conditions as in Theorem 3.11, we have
\begin{eqnarray*}
 (\frac{\partial}{\partial t}+\triangle)\hat{e}=-2|\nabla_{A}(\wedge F_{A}+\mu(u)-c)|^{2}-2|L_{u}(\wedge F_{A}+\mu(u)-c)|^{2}. \end{eqnarray*}
\end{lemma}
\begin{proof} By direct computation, we have \begin{eqnarray*}
\frac{1}{2}\frac{\partial}{\partial t}|\wedge F_{A}+\mu(u)-c|^{2}&=&Re<\frac{\partial}{\partial t}(\wedge F_{A})+\frac{\partial}{\partial t}(\mu(u)-c), \wedge F_{A}+\mu(u)-c>\\
&=& Re<\frac{\partial}{\partial t}(\wedge F_{A}), \wedge F_{A}+\mu(u)-c>\\
&&+Re<\frac{\partial}{\partial t}(\mu(u)-c), \wedge F_{A}+\mu(u)-c>.
\end{eqnarray*}
And we also have
\begin{eqnarray*}
&&Re<\frac{\partial}{\partial t}(\wedge F_{A}), \wedge F_{A}+\mu(u)-c>\\
&=&Re<\wedge d_{A}(\frac{\partial A}{\partial t}),\wedge F_{A}+\mu(u)-c>\\
&=&Re<-\wedge d_{A}(L^{*}_{u}d_{A}u+d^{*}_{A}F_{A}), \wedge F_{A}+\mu(u)-c>((5))\\
&=&Re<i\wedge d_{A}({\bar \partial}_{A}-{\partial}_{A})\wedge F_{A}-\wedge d_{A}(L^{*}_{u}d_{A}u)), \wedge F_{A}+\mu(u)-c> (\textit{lemma }3.6)\\
&=&Re<i\wedge({\partial}_{A}{\bar \partial}_{A}-{\bar \partial}_{A}{\partial}_{A})\wedge F_{A},  \wedge F_{A}+\mu(u)-c>\\
&&-Re<\wedge d_{A}(L^{*}_{u}d_{A}u)), \wedge F_{A}+\mu(u)-c> ({\bar \partial}^{2}_{A}=0)\\
&=&-Re<\nabla^{*}_{A}\nabla_{A}(\wedge F_{A}), \wedge F_{A}+\mu(u)-c>\\
&&-Re<\wedge d_{A}(L^{*}_{u}d_{A}u)), \wedge F_{A}>-Re<\wedge d_{A}(L^{*}_{u}d_{A}u)), \mu(u)-c> (\textit{lemma }3.7)\\
&=&-Re<\nabla^{*}_{A}\nabla_{A}(\wedge F_{A}), \wedge F_{A}+\mu(u)-c>\\
&&-Re<\wedge d_{A}(d_{A}u), L_{u}(\wedge F_{A})>-Re<\wedge d_{A}(d_{A}u), L_{u}(\mu(u)-c)>\\
&=&-Re<\nabla^{*}_{A}\nabla_{A}(\wedge F_{A}), \wedge F_{A}+\mu(u)-c>-|L_{u}(\wedge F_{A})|^{2}\\
&&-Re<L_{u}(\wedge F_{A}), L_{u}(\mu(u)-c)> (d_{A}^{2}=F_{A}),
\end{eqnarray*}
and
\begin{eqnarray*}
&&Re<\frac{\partial}{\partial t}(\mu(u)), \wedge F_{A}+\mu(u)-c>\\
&=&-Re<d\mu(\nabla^{*}_{A}d_{A}u), \wedge F_{A}+\mu(u)-c>\\
&&-Re<d\mu((d\mu)^{*}(\mu(u)-c), \wedge F_{A}+\mu(u)-c>((6))\\
&=&-Re<d\mu(\nabla^{*}_{A}d_{A}u), \wedge F_{A}+\mu(u)-c>\\
&&-Re<L^{*}_{u}L_{u}(\mu(u)-c), \wedge F_{A}+\mu(u)-c>( \textit{lemma }2.3)\\
&=&-Re<d\mu(\nabla^{*}_{A}d_{A}u), \wedge F_{A}+\mu(u)-c>\\
&&-Re<L_{u}(\mu(u)-c), L_{u}(\wedge F_{A})>-|L_{u}(\mu(u)-c)|^{2}.
\end{eqnarray*}
Then add equalities above, we have \begin{eqnarray*}
\frac{1}{2}\frac{\partial}{\partial t}\hat{e}&=&-Re<\nabla^{*}_{A}\nabla_{A}(\wedge F_{A}), \wedge F_{A}+\mu(u)-c>\\
&&-Re<d\mu(\nabla^{*}_{A}d_{A}u), \wedge F_{A}+\mu(u)-c>-|L_{u}(\mu(u)-c)|^{2}\\
&&-|L_{u}(\wedge F_{A})|^{2}-2Re<L_{u}(\wedge F_{A}), L_{u}(\mu(u)-c)>\\
&=&-Re<\nabla^{*}_{A}\nabla_{A}(\wedge F_{A}), \wedge F_{A}+\mu(u)-c>\\
&&-Re<d\mu(\nabla^{*}_{A}d_{A}u), \wedge F_{A}+\mu(u)-c>-|L_{u}(\wedge F_{A}+\mu(u)-c)|^{2}.
\end{eqnarray*}
On the other hand, we have\begin{eqnarray*}
\frac{1}{2}\triangle \hat{e}&=&\nabla^{*}_{A}Re<\nabla_{A}(\wedge F_{A}+\mu(u)-c), \wedge F_{A}+\mu(u)-c)>\\
&=&Re<\nabla^{*}_{A}\nabla_{A}(\wedge F_{A}+\mu(u)-c), \wedge F_{A}+\mu(u)-c)>\\
&&-|\nabla_{A}(\wedge F_{A}+\mu(u)-c)|^{2}\\
&=&Re<\nabla^{*}_{A}\nabla_{A}(\wedge F_{A}), \wedge F_{A}+\mu(u)-c)>\\
&&+Re<\nabla^{*}_{A}\nabla_{A}(\mu(u)), \wedge F_{A}+\mu(u)-c)>\\
&&-|\nabla_{A}(\wedge F_{A}+\mu(u)-c)|^{2}\\
&=&Re<\nabla^{*}_{A}\nabla_{A}(\wedge F_{A}), \wedge F_{A}+\mu(u)-c)>\\
&&+Re<d\mu(\nabla^{*}_{A}d_{A}u), \wedge F_{A}+\mu(u)-c)>\\
&&-|\nabla_{A}(\wedge F_{A}+\mu(u)-c)|^{2}.
\end{eqnarray*}
Combining equalities above we complete the proof of lemma 3.12. \end{proof}
\begin{lemma} Assume the same conditions as in Theorem 3.11, then
${sup}_{X}\hat{e}$ is a decreasing function of $t$.
\end{lemma}
\begin{proof}This follows from Theorem 3.11 and the maximum principal for the heat operator $(\frac{\partial}{\partial t})+\triangle$ on $X$. \end{proof}

Then using this lemma, we can prove the following corollary similar to [8, Lemma 8].\\
\begin{corollary} Let $H(t)=H_{0}h(t)$ where $h(t)$ is smooth solution of (18) in $X\times [0, T)$ with the initial value $H_{0}$ for a finite time $T>0$. Then $H(t)$ converges in $C^{0}$ to a nondegenerate continuous metric $H(T)$ as $t\rightarrow T$. \end{corollary}
Using almost the same procedure as [6] (or [8]), we can further prove
\begin{lemma} Let $H(t), 0\leq t< T$, be any one-parameter family of Hermitian metrics on a holomorphic principal bundle $P$ over a compact K$\ddot{\textmd{a}}$hler manifold $X$ such that\\
(i) $H(t)$ converges in $C^{0}$ to some continuous metric $H(T)$ as $t\rightarrow T$,\\
(ii) ${sup}_{X}|{\hat{F}}_{H}|$ is uniformly bounded for $t< T$.\\
Then $H(t)$ is bounded in $C^{1,\alpha}$ independently of $t\in[0, T)$ for $0<\alpha<1$.
\end{lemma}
Now we can prove the global existence of (17) in the following theorem.\\
\begin{theorem} Assume that both the base manifold $X$ and the fiber $M$ are compact K$\ddot{\textmd{a}}$hler manifolds, where $M$ supports a Hamiltonian $G$-group action, and $P$ is a holomorphic G-principal bundle over $X$. Let $H_{0}$ be a Hermitian metric on $P$, $A_{0}$ the canonical connection with respect to $H_{0}$ (it is well-known that $ A_{0}\in\mathcal{A}^{1,1}$), $u_{0}$ a holomorphic section of the fiber bundle $P\times_{G}M$, i.e. ${\bar \partial}_{A_{0}}u_{0}=0$. Then the equation
\begin{eqnarray*} \frac{\partial H}{\partial t}=-2iH[\wedge F_{H}+\mu(u)-c] \end{eqnarray*}
has a unique solution $H(t)$ which exists on $X$ for $0\leq t<\infty$.
\end{theorem}
\begin{proof} The proof is totally similar to [8], so here we only give a sketch. The local existence and uniqueness of this equation has been proved by previous result. Suppose that the solution exists for $0\leq t< T$. By
 lemma 3.13, ${sup}_{X}|\hat{F_{H}}|^{2}$ is bounded independently of $t\in [0, T)$ and ${\hat{F}}_{H}$ is bounded independently of $t\in [0, T)$. By corollary 3.14, $H(t)$ converges in $C^{0}$ to a nondegenerate continuous limit metric $H(T)$ as $t\rightarrow T$. Thus by lemma 3.15, $H(t)$ is bounded in $C^{1,\alpha}$ independently of $t\in [0, T)$.
  Let $H(t)=H_{0}h(t)$, where $h(t)$ solves
 \begin{eqnarray*}
\frac{\partial h}{\partial t}=-\triangle_{A_{0}}h-i[h\hat{F}_{A_{0}}+\hat{F}_{A_{0}}h+2(\mu(u_{0})-c)h]+2i\wedge({\bar \partial}_{A_{0}}h h^{-1}{\partial}_{A_{0}}h). \end{eqnarray*}
Then it is totally the same as the proof of [8, Theorem 11] to prove that $h$ is $C^{2,\alpha}$ and $\frac{\partial h}{\partial t}$ is $C^{\alpha}$ with bounds independent of $t\in [0, T)$. Thus $H(t)\rightarrow H(T)$ in $C^{2,\alpha}$, hence in $C^{\infty}$. For the initial value $H(T)$ we use local existence theorem (Theorem 3.8) again. Therefore the solution continues to exist for $t<T+\epsilon$, for some $\epsilon$. This completes the proof. \end{proof}

\begin{proof}[Proof of Theorem 1.1] By Theorem 3.8, Theorem 3.10 and Theorem 3.16, we can immediately obtain Theorem 1.1. \end{proof}

\section{Application}
$\hspace{15pt}$In this section we will use the heat flow of Yang-Mills-Higgs functional to give a new proof of Theorem 1.2 which is critical to his Hitchin-Kobayashi correspondence theorem [13, 14]. The idea is the same as [8]'s approach to combine Donaldson's method [6] with Uhlenbeck and Yau's method [19]. Due to the property of the heat flow, we apply a well-known PDE result to prove an inequality which relates $C^{0}$ and $L^{1}$ norms of $s\in {Met}^{p}_{2,B}$, that is \begin{equation}
sup|s|\leq C_{B} ||s||_{L^{1}}. \end{equation} Then we prove \begin{equation} ||s||_{L^{1}}\leq C_{1}'\Psi^{c}(e^{s})+C'_{2}. \end{equation}
After combining the two inequalities, we can prove Theorem 1.2.

\subsection{c-stability}
$\hspace{15pt}$First let's recall the definition of $c$-stability [14].
Suppose that $(M, \omega, \mu)$ is a compact K$\ddot{\textmd{a}}$hler manifold with a Hamiltonian group $K$-action and $\mu: M\rightarrow \mathfrak{t}$ is the corresponding moment map, where the Lie algebra $\mathfrak{t}$ carries an invariant inner product $<\cdot,\cdot>$. We can extend the action of $K$ on $M$ to the complexification $G$ of $K$ if $M$ is a compact K$\ddot{\textmd{a}}$hler manifold.\\

$\hspace{-17pt}$$\textbf{Definition 4.1.}$  $\forall x\in M$ and $s\in\mathfrak{t}$. Let $\lambda_{t}(x;s)=\mu_{s}(e^{its}x)=<\mu(e^{its}x),s>$. Define the $maximal$ $weight$ $\lambda(x;s)$ of the action of $s$ on $x$ to be
\begin{eqnarray*} \lambda(x;s)=lim_{t\rightarrow\infty}\lambda_{t}(x;s)\in \mathbb{R}\cup\{\infty\}. \end{eqnarray*} \par
Let $\mathfrak{g}=Lie(G)$, we have $\mathfrak{g}=\mathfrak{l}\oplus \mathfrak{g}^{s}$, where $\mathfrak{l}$ is the center of $\mathfrak{g}$, and
$\mathfrak{g}^{s}=[\mathfrak{g},\mathfrak{g}]$. Suppose that $\mathfrak{h}\subset \mathfrak{g}^{s}$ is a Cartan subalgebra , and $R\subset\mathfrak{ h}^{*}$ is the set of roots, then we have
$\mathfrak{g}=\mathfrak{l}\oplus \mathfrak{h} \oplus \bigoplus_{\alpha\in R}\mathfrak{g}_{\alpha},$
where $\mathfrak{g}_{\alpha}\subset \mathfrak{g}^{s}$ is the subspace on which $\mathfrak{h}$ acts through the character $\alpha\in \mathfrak{h}^{*}$.
We can decompose the set $R$ of roots in positive and negative roots as: $R=R^{+}\cup R^{-}$. Write the set of simple roots $\triangle=(\alpha_{1},\ldots,\alpha_{r})\subset R^{+}$. Then take any subset $A=\{\alpha_{i_{1}},\ldots,\alpha_{i_{s}}\}\subset\triangle$, and let \\

$D=D_{A}=\{\alpha\in R|\alpha=\Sigma_{j=1}^{r}m_{j}\alpha_{j}$, where $m_{i_{t}}\geq0$ for $1\leq t \leq s\}$.
\\

$\hspace{-17pt}$$\textbf{Definition 4.2.}$ We call the subalgebra $\mathfrak{q}=\mathfrak{l}\oplus\mathfrak{ h}\oplus\bigoplus_{\alpha\in D}\mathfrak{g}_{\alpha}$ $parabolic$ $subalgebra$ of $\mathfrak{g}$ with respect to $A\subset\triangle$. The connected subgroup $Q$ of $G$ whose subalgebra is $\mathfrak{q}$ is called the $parabolic$ $subgroup$ of $G$ with respect to $A$. Let $\lambda_{1},\ldots,\lambda_{r}$ be the set of fundamental weights, then we call any positive (resp. negative) linear combination of the fundamental weights $\lambda_{i_{1}},\ldots, \lambda_{i_{s}}$ plus an element of the dual of $i(\mathfrak{l}\cap \mathfrak{k})$ a $dominant$ (resp. $antidominant$) $character$ on $\mathfrak{q}$ (or on $Q$).\\

$\hspace{-17pt}$$\textbf{Definition 4.3.}$
Let $P\rightarrow X$ be principal $G$-bundle, $P_{G}=P\times_{K}G$, and $P(G/Q)=P_{G}\times_{G}(G/Q)$. We call the space $\Gamma(P(G/Q))$ of sections of the bundle $P(G/Q)$ as the space of $reductions$ of the structure group of $P_{G}$ from $G$ to $Q$.  A reduction $\sigma$ is $holomorphic$ if the map $\sigma: X\rightarrow P(G/Q)$ is holomorphic. \\

According to [14], there is a section $g_{\sigma,\chi}\in\Omega^{0}(P\times_{Ad}i\mathfrak{k})$ which is fiberwise the dual of $\chi$ for a parabolic subgroup $Q\subset G$, a reduction $\sigma\in \Gamma(P(G/Q))$ and an antidominant character $\chi$. [14] has also proved that $\mathcal{ L}$ is a K$\ddot{\textmd{a}}$hler manifold with a Hamiltonian $\mathcal{G}_{K}(=\Gamma(P\times_{Ad}K)$-group action, its corresponding moment map is
\begin{eqnarray*} \mu_{\mathcal{L}}(u)=\mu(u)=\mu\cdot u. \end{eqnarray*}
Therefore, we can define the maximal weight of $s\in Lie(\mathcal{G}_{K})=\Omega^{0}(P\times_{Ad}\mathfrak{k})$ acting on a section $u\in\mathcal{ L}$ as
\begin{eqnarray*}
\int_{x\in X}\lambda(u(x);s(x)).
\end{eqnarray*}
Fixing any central element $c\in \mathfrak{l}\cap \mathfrak{k}$, $\forall u\in\mathcal{ L}$, we define the $c-$$total$ $degree$ of the pair $(\sigma, \chi)$ as follows:
\begin{eqnarray*}
T_{u}^{c}(\sigma,\chi)=deg(\sigma,\chi)+\int_{x\in X}\lambda(u(x);-ig_{\sigma,\chi}(x))+<i\chi,c>Vol(X),
\end{eqnarray*}
where the definition of $deg(\sigma,\chi)$ can be found in [14, 2.8].\\

$\hspace{-17pt}$$\textbf{Definition 4.4.}$  A pair $(A, u)\in \mathcal{A}^{1,1}\times \mathcal{L}$ is called $c$-$stable$, if for any $X_{0}\subset X$ whose complement $X\setminus X_{0}$ is a complex codimension 2 subvariety of $X$, for any parabolic subgroup $Q\subset G$, for any antidominant character $\chi$ of $Q$, and for any holomorphic reduction $\sigma\in\Gamma(X_{0};P(G/Q))$, we have
\begin{eqnarray*} T_{u}^{c}(\sigma,\chi)>0. \end{eqnarray*}

$\hspace{-17pt}$$\textbf{Definition 4.5.}$ A pair $(A, u)$ is called $simple$ if there is no semisimple element in $Lie(\mathcal{G}_{G})$ which leaves $(A, u)$ fixed, where $\mathcal{G}_{G}=\Gamma(P\times_{Ad}G)$.\\

Now we can state the main theorem of [14] as follows:
\begin{theorem}[Hitchin-Kobayashi correspondence] Assume that $(A, u)\in \mathcal{A}^{1,1}\times \mathcal{L}$ is a simple pair. Then there exists a gauge transformation $g\in \mathcal{G}_{G}$ such that
\begin{eqnarray*} \wedge F_{g(A)}+\mu(g(u))=c, \end{eqnarray*}
if and only if $(A, u)$ is c-stable.
\end{theorem}
\subsection{Heat flow proof of Mundet's theorem}
$\hspace{15pt}$According to [14], both $\mathcal{ A}^{1,1}$ and $\mathcal{L}$ are K$\ddot{\textmd{a}}$hler manifolds. Hence $\mathcal{ A}^{1,1}\times \mathcal{L}$ is also a K$\ddot{\textmd{a}}$hler manifold with a Hamiltonian $\mathcal{G}_{K}$-action, and the moment map is $\mu^{c}(A, u)=\wedge F_{A}+\mu(u)-c$. Then we have the integral of the moment map $\Psi^{c}$ (for more details see [14, section 3]) defined on $\mathcal{ A}^{1,1}\times \mathcal{L}\times \mathcal{G}_{G}$. We will see that if the pair $(A, u)\in\mathcal{ A}^{1,1}\times \mathcal{L}$ is $c$-stable, then the map $\Psi^{c}$ satisfies an inequality like that in [4, 16]. This method is almost the same as that appears in [8]. Therefore, we only give a sketch in some steps of the proof, referring to [8] (or [4]) for details.\par
Recall that there is an invariant inner product $<,>$ on $\mathfrak{g}$ which induces a norm $|\cdot|$. On $\Omega^{0}(P\times_{Ad}\mathfrak{g})$ we can define the $L^{p}$ norm:
\begin{eqnarray*}
||s||_{L^{p}}=(\int_{X}|s(x)|^{p})^{\frac{1}{p}},
\end{eqnarray*}
and the $L_{2}^{p}$ norm:
\begin{eqnarray*}
||s||_{L_{2}^{p}}=||s||_{L^{p}}+||d_{A}s||_{L^{p}}+||\nabla d_{A}s||_{L^{p}},
\end{eqnarray*}
where $d_{A}$ is the covariant derivative with respect to the connection $A$, and
\begin{eqnarray*}
\nabla=\nabla^{LC}\otimes d_{A}: \Omega^{0}(T^{*}X\otimes P\times_{Ad}\mathfrak{g})\rightarrow \Omega^{1}(T^{*}X\otimes P\times_{Ad}\mathfrak{g}),
\end{eqnarray*}
where $\nabla^{LC}$ is the Levi-Civita connection. Now we can define
\begin{eqnarray*}{\mathcal{ M}et}_{2}^{p}=L_{2}^{p}(P\times_{Ad}i\mathfrak{k}) \end{eqnarray*}
as the closure of $\Omega^{0}(P\times_{Ad}i\mathfrak{k})$ with respect to the norm $||\cdot||_{L_{2}^{p}}$. Then take a subset of ${\mathcal{ M}et}_{2}^{p}$ as follows:
\begin{eqnarray*}{\mathcal{ M}et}_{2,B}^{p}=\{s\in {\mathcal{ M}et}_{2}^{p}|||\mu^{c}(e^{s}(A, u))||_{L^{p}}^{p}\leq B\}, \end{eqnarray*}
where $B\in \mathbb{R}^{+}$.\par
Next using previous theorems on the global existence of heat flow in section 3, we can prove a lemma totally similar to [8, lemma 12]:
\begin{lemma} Let $H(t)=H_{0}h(t)=H_{0}e^{s(t)}$ be Hermitian metrics which $h$ is the global solution of (18) on $X\times [0, +\infty)$. Then for any sequence of $t_{j}\rightarrow \infty$, Tr$(\wedge F_{H}(t_{j})+\mu(e^{s(t_{j})}u_{0}))$ converges to some constant $\alpha$ in weakly $L^{2}_{1}$. Moreover, $\int_{X}<\mu(e^{s(t_{j})}u_{0}), s(t_{j})>$ also converges to a constant as $t_{j}\rightarrow \infty$.
\end{lemma}
\begin{proof} By lemma 3.1 (energy inequality), we know
\begin{equation}
\int|L^{*}_{u(t_{j})}d_{A(t_{j})}u(t_{j})+D^{*}_{A(t_{j})}F_{A(t_{j})}|^{2}\rightarrow 0,
\end{equation}
as $t_{j}\rightarrow\infty$. By lemma 3.5 and lemma 3.6 we have
\begin{equation}
i(-{\partial}_{A}+{\bar \partial}_{A})(\wedge F_{A}+\mu(u))=-L^{*}_{u}d_{A}u-D^{*}_{A}F_{A}. \end{equation}
Therefore \begin{eqnarray*}
\int|\nabla_{A}(\wedge F_{A(t_{j})}+\mu(u(t_{j})))|^{2}
&=&\int|L^{*}_{u(t_{j})}d_{A(t_{j})}u(t_{j})+D^{*}_{A(t_{j})}F_{A(t_{j})}|^{2}\rightarrow 0. \end{eqnarray*}
Since $\nabla_{A} Tr=Tr \nabla_{A}$, we have
\begin{equation}
\int|\nabla_{A}Tr(\wedge F_{A(t_{j})}+\mu(u(t_{j})))|^{2}\rightarrow0.
\end{equation}
By Theorem 3.11, we know that $|\wedge F_{A(t_{j})}+\mu(u(t_{j}))|$ is bounded and then
\begin{eqnarray*}
Tr(\wedge F_{H}(t_{j})+\mu(e^{(t_{j})}u_{0}))=Tr(\wedge F_{A(t_{j})}+\mu(u(t_{j})))\rightarrow\alpha, \end{eqnarray*}
weakly in $L^{2}_{1}$ as $t_{j}\rightarrow \infty$ for some function $\alpha$ with $\nabla\alpha=0$. This means that $\alpha$ is a constant. By Chern-Weil theory, we have
\begin{eqnarray*}
2i\int Tr(\wedge F_{H}+\mu(e^{s}u_{0}))=4\pi C_{1}(P)+2i\int<\mu(e^{s}u_{0}), s>. \end{eqnarray*}
Then \begin{equation}
{lim}_{t_{j}\rightarrow\infty}\int<\mu(e^{s(t_{j})}u_{0}),s(t_{j})>=\alpha Vol(X)+2i\pi C_{1}(P).
\end{equation}
This completes the proof.
\end{proof}
By a result of Donaldson (see [4, lemma 3.7.2]), Mundet has proved the following lemma in [14].
\begin{lemma} Let $H(t)=H_{0}h(t)=H_{0}e^{s(t)}$ be Hermitian metrics where $h(t)$ is the solution of (18), then there exists a constant $C_{B}$ such that
\begin{equation}
sup|s|\leq C_{B}||s||_{L^{1}}.
\end{equation}
\end{lemma}

\begin{proof}[Proof of Theorem 1.2] Using lemma 4.3, and following Simpson's method [8, 16], we can prove Theorem 1.2.\\

For $H=H_{0}h=H_{0}e^{s}$, define the integral of the moment map (see [14] for more details) as follows:
\begin{eqnarray*}
\Psi^{c}(e^{s})&=&M_{u_{0},c}(H(t))\\&=&2i\int_{X}Tr(s\wedge F_{H_{0}})dX+2\int_{X}<\phi(s){\bar\partial}_{A_{0}}s, {\bar\partial}_{A_{0}}s>dX\\
&&+2i\int_{X}<\mu(e^{s}u_{0})-\mu(u_{0}), s>dX-2i\int_{X}\textit{log }\textit{det }(e^{s}\cdot c)dX,
\end{eqnarray*}
where $\phi$ is constructed from the function
\begin{eqnarray*} \phi(\lambda_{1}, \lambda_{2})=\frac{e^{\lambda_{2}-\lambda_{1}}-(\lambda_{2}-\lambda_{1})-1}{(\lambda_{2}-\lambda_{1})^{2}} \end{eqnarray*}
The functional can be seen as a modified Donaldson functional. In fact, when $M=\{pt\}$, it coincides with the Donaldson functional [6].\par
Then we have
\begin{eqnarray}
\frac{d}{dt}\Psi^{c}(e^{s})=\frac{d}{dt}M_{u_{0},c}(H(t))=-4\int_{X}|\wedge F_{H}+\mu(u)-c|^{2}dX\leq0.
\end{eqnarray}
Since $M_{u_{0},c}(H(0))=0$, it yields that
\begin{eqnarray}
\Psi^{c}(e^{s})=M_{u_{0},c}(H(t))=-4\int_{0}^{t}\int_{X}|\wedge F_{H}+\mu(u)-c|^{2}dXd\tau\leq0.
\end{eqnarray}
By lemma 4.3, to prove Theorem 1.2, we only need prove
\begin{eqnarray}
||s||_{L^{1}}\leq C_{1}'\Psi^{c}(e^{s})+C_{2}'.
\end{eqnarray}
We can suppose the contrary and then deduce that the pair $(A, u)$ cannot be $c$-stable in this case. If there exists not such constants, then we can find a sequence of $C_{j}\rightarrow \infty$ and $s_{j}\in {\mathcal{M}et}_{2, B}^{p}$ with $||s_{j}||_{L^{1}}\rightarrow\infty$ such that $||s_{j}||_{L^{1}}\geq C_{j}\Psi^{c}(e^{s})$. Let $l_{j}=||s_{j}||_{L^{1}}$, $u_{j}=l_{j}^{-1}s_{j}$, so $||u_{j}||_{L^{1}}=1$ and $sup|u_{j}|\leq C$. Then we have
\begin{lemma}([14]) After passing to a subsequence, there exists $u_{\infty}\in L_{1}^{2}(P\times_{Ad}i\mathfrak{g})$ such that $u_{j}\rightharpoonup u_{\infty}$ in $L_{1}^{2}(P\times_{Ad}i\mathfrak{g})$ and
\begin{eqnarray*} \lambda((A, u); -iu_{\infty})\leq0. \end{eqnarray*}
\end{lemma}
\begin{proof} Using the proposition (2) of the integral of the moment map in [14, Proposition 3.3], we have
\begin{eqnarray*}
\frac{d}{dt}\Psi^{c}(e^{s})=\frac{d}{dt}M_{u_{0},c}(H(t))=\lambda_{t}((A,u);-is).
\end{eqnarray*}
By (44) we can obtain
\begin{eqnarray*}
\lambda_{t}((A,u);-iu_{j})\leq0.
\end{eqnarray*}
In particular $\lambda_{t}((A,u);-iu_{\infty})\leq0, \forall t>0$. By Definition 4.1, we have
\begin{eqnarray*}
\lambda((A,u);-iu_{\infty})\leq0.
\end{eqnarray*}
This completes the proof of lemma 4.4. \end{proof}

Following [6] or [8], then using lemma 4.2, the same argument in [16, lemma 5.5] yields that all eigenvalues $\lambda_{l}$ of $u_{\infty}$ are constants for almost $x\in X$, and $u_{\infty}$ defines a filtration of $V$ by holomorphic subbundles in the complement of a complex codimension 2 subvariety $X_{0}$ of $X$. The filtration of $V$ on $X_{0}$ and $u_{\infty}$ lead to a holomorphic reduction $\sigma \in\Gamma(X_{0}; P(G/Q))$ for some parabolic subgroup $Q$ of $G$, and an antidominant character $\chi$ of $Q$. By [14, lemma 4.3], the $c$-total degree $T_{u}^{c}(\sigma,\chi)$ of the pair $(\sigma,\chi)$ equals $\lambda((A,u);-iu_{\infty})$. Then by lemma 4.4 we have
\begin{eqnarray*}T_{u}^{c}(\sigma,\chi)\leq0. \end{eqnarray*}
But this contradicts the $c$-stability condition. Thus we complete the proof of Theorem 1.2.
\end{proof}

\section{Acknowledgements}
$\hspace{15pt}$ The author would like to thank Professor Gang Tian and Professor Huijun Fan for constant guidance and encouragement. \\

College of Science, National University of Defense Technology, Changsha 410073, Hunan, P. R. China\\
Email address: aijinlin@pku.edu.cn
\end{document}